\documentclass[a4paper,12pt,twoside]{article}
\pdfoutput=1
\usepackage{graphicx}
\usepackage{color}
\usepackage{pict2e}
\usepackage{dsfont}

\newcommand{\pictype}{} 

\def\PaperTypeInput0

\ifthenelse{\PapierTypeInput = 1}
{\setlength\paperheight{11in}%
 \setlength\paperwidth{8.5in}}
{\setlength\paperheight{297mm}%
 \setlength\paperwidth{210mm}}

\renewcommand{\baselinestretch}{1.21}
\makeatletter
\@addtoreset{equation}{section}
\@addtoreset{figure}{section}
\makeatother

\newtheorem{prop}{\bf Proposition}[section]

\newtheorem{theorem}[prop]{\bf Theorem}
\newtheorem{remark}[prop]{\bf Remark}
\newenvironment{proof}
  {\begin{trivlist}\item[]{\bf Proof.}}
  {\hspace*{\fill}{$\bowtie$}\end{trivlist}}

\newcommand{\setR}{\mathds{R}}

\newcommand{\imagi}{\mathrm{i}}
\newcommand{\spec}{\mathrm{spec\,}}
\newcommand{\const}{\mathrm{const.}}

\newcommand{\hot}{\mathrm{h.o.t.}}
\newcommand{\Ord}{\mathcal{O}}
\newcommand{\Diff}{\mathrm{D}}
\newcommand{\diff}{\mathrm{d\,}}

\begin{document}
\pagestyle{empty}
\markboth
  {Stefan Liebscher}
  {Codim-one manifolds of equilibria and bifurcation without parameters}
\pagenumbering{arabic}

\vspace*{\fill}

\begin{center}
\Large\bfseries
Dynamics near manifolds of equilibria\\
of codimension one \\
and bifurcation without parameters
\end{center}

\vspace*{\fill}

\begin{center}
\textbf{\large Stefan Liebscher}
\\ ~
\\ Free University Berlin
\\ Institute of Mathematics
\\ Arnimallee 3
\\ D--14195 Berlin
\\ Germany
\\ ~
\\ {\tt stefan.liebscher@fu-berlin.de}
\\ {\tt http://dynamics.mi.fu-berlin.de/}
\end{center}

\vspace*{\fill}

\begin{center}
Preprint, 2010-Sep-07
\end{center}

%

\vspace*{\fill}

\clearpage

\vspace*{\fill}

\begin{abstract}
\renewcommand{\baselinestretch}{1.15}
\setlength{\parskip}{1ex}
\setlength{\parindent}{0mm}
We investigate the breakdown of normal hyperbolicity of a manifold of equilibria of a flow.
In contrast to classical bifurcation theory we assume the absence of any flow-invariant 
foliation at the singularity transverse to the manifold of equilibria. 
We call this setting bifurcation without parameters.

In the present paper we provide a description of general systems with a manifold 
of equilibria of codimension one as a first step 
towards a classification of bifurcations without parameters.
This is done by relating the problem to singularity theory of maps.

\bigskip


\end{abstract}

\vspace*{\fill}

\cleardoublepage
\setcounter{page}{1}
\pagestyle{myheadings}

\section{Introduction}
\label{Lie:secIntroduction}

We study dynamical systems with manifolds of equilibria near points 
at which normal hyperbolicity of these manifolds is violated.
Manifolds of equilibria appear frequently in classical bifurcation theory 
by continuation of a trivial equilibrium.
Here, however, we are interested in manifolds of equilibria 
which are not caused by additional parameters.
In fact we require the absence of any flow-invariant foliation 
transverse to the manifold of equilibria at the singularity.
We therefore call the emerging theory \emph{bifurcation without parameters}.

Albeit the apparent degeneracy of our setting 
(of infinite codimension in the space of all smooth vectorfields)
there is a surprisingly rich and diverse set of applications ranging from
networks of coupled oscillators \cite{Liebscher97-Diplom},
viscous and inviscid profiles of stiff hyperbolic balance laws \cite{HaerterichLiebscher05-TravWaves},
standing waves in fluids \cite{AfendikovFiedlerLiebscher07-PlaneKolmogorovFlows,
AfendikovFiedlerLiebscher08-PlaneKolmogorovFlows},
binary oscillations in numerical discretizations \cite{FiedlerLiebscherAlexander98-HopfBinOsc},
population dynamics \cite{Farkas84-ZipBifurcation}, 
cosmological models \cite{Wainwright05-Cosmology},
and many more.
The present paper is a first step towards a classification 
of bifurcations without parameters.

Consider a vector field
\begin{equation}\label{Lie:eqGeneralVectorFieldNxM}
\begin{array}{rcll}
 \dot{x} &=& f(x,y) \qquad& \in\setR^n, \\
 \dot{y} &=& g(x,y) \qquad& \in\setR^m
\end{array}
\end{equation}
with a manifold of equilibria $\{\,(0,y)\;:\;y\in\setR^m\,\}$, i.e.
\begin{equation}\label{Lie:eqGeneralEqManifold}
 f(0,y) \;\equiv\; 0, \qquad g(0,y) \;\equiv\; 0.
\end{equation}
As long as the manifold remains normally hyperbolic, 
i.e. the linearization of $f$ on the manifold has no purely imaginary eigenvalues,
\begin{equation}\label{Lie:eqGeneralHyperbolicity}
 \spec \partial_x f(0,y) \,\cap\, \imagi\setR \;=\; \emptyset,
\end{equation}
there exists a local flow-invariant foliation with leaves homeomorphic 
to a standard saddle, for example by the theorem of Shoshitaishvili 
\cite{Shoshitaishvili75-Bif}.
Bifurcations are characterized by a non-hyperbolic block $A$ of the linearization
\begin{equation}\label{Lie:eqGeneralLinearization}
\left(\begin{array}{cc}A(y)&0\\B(y)&0\end{array}\right) \;=\;
\left(\begin{array}{cc} \partial_x f & \partial_y f \\ \partial_x g & \partial_y g
      \end{array}\right)(0,y),
\end{equation}
say at the origin, i.e. the spectrum of $A(0)$ intersects the imaginary axis,
\begin{equation}\label{Lie:eqGeneralNonHyperbolicity}
 \spec A(0) \,\cap\, \imagi\setR \;\neq\; \emptyset.
\end{equation}
Restricting to a center manifold we can assume
\begin{equation}\label{Lie:eqGeneralCenter}
 \spec A(0) \;\subset\; \imagi\setR.
\end{equation}
We will always assume that the vector field is smooth enough 
to allow suitable expansions.

Note the analogy to classical bifurcation theory where $y$ would be a parameter,
i.e. $g \equiv 0$.
For references on classical bifurcation theory see for example
\cite{Arnold83-GeometricalMethods, 
HaleKocak91-Dynamics, Kuznetsov95-Bifurcation, Vanderbowhede89-CentManifolds} 
and the references there.
In the classical case, $g \equiv 0$, however, the flow invariant transverse foliation 
$\{ y = \const \}$ is also present in a neighborhood of the bifurcation point.
This is no longer true in the general case 
(\ref{Lie:eqGeneralVectorFieldNxM}, \ref{Lie:eqGeneralEqManifold})
without parameters.
Indeed, generic functions $g$ of the form (\ref{Lie:eqGeneralEqManifold}) yield 
a drift in the ``parameter'' direction $y$ 
which excludes any flow-invariant foliation transverse to the manifold of equilibria 
near a singularity (\ref{Lie:eqGeneralNonHyperbolicity}).
Thus, the resulting nonlinear local dynamics differ considerably 
from classical bifurcation scenarios.

A rigorous analysis of bifurcations without parameters
(\ref{Lie:eqGeneralVectorFieldNxM}, \ref{Lie:eqGeneralEqManifold},
 \ref{Lie:eqGeneralCenter})
has been carried out for the following cases:
\begin{enumerate}
\item
A simple eigenvalue zero of $A$ (transcritical point), 
$n=m=1$,
with linearization at the bifurcation point:
\[
\left( \begin{array}{cc} 0 & 0 \\ 1 & 0 \end{array} \right),
\]
see \cite{Liebscher97-Diplom}.
\item
A pair of purely imaginary nonzero eigenvalues of $A$ (Andronov-Hopf point), 
$n=2$, $m=1$,
with linearization at the bifurcation point:
\[
\left( \begin{array}{ccc} 0 & 1 & 0 \\ -1 & 0 & 0 \\ 0 & 0 & 0 \end{array} \right),
\]
see \cite{FiedlerLiebscherAlexander98-HopfTheory}. 
A partial description can also be found in \cite{Farkas84-ZipBifurcation}.
\item
An algebraically double and geometrically simple eigenvalue zero of $A$ (Bogdanov-Takens point),
$n=m=2$,
with linearization at the bifurcation point:
\[
\left( \begin{array}{cccc} 0 & 0 & 0 & 0 \\ 1 & 0 & 0 & 0 \\ 
       0 & 1 & 0 & 0 \\ 0 & 0 & 0 & 0 \end{array} \right),
\]
see \cite{FiedlerLiebscher01-TakensBogdanov}. 
With additional symmetries, 
this case is also studied in \cite{AfendikovFiedlerLiebscher07-PlaneKolmogorovFlows,
AfendikovFiedlerLiebscher08-PlaneKolmogorovFlows} 
\end{enumerate}
Note the nonzero blocks $B$ at transcritical and Bogdanov-Takens points
yielding a drift in $y$-direction and excluding any flow-invariant transverse foliation.
At Andronov-Hopf points this drift is induced by a generic second-order term
$\Pi_y\Delta_x({f \atop g})(0) \neq 0$ which is the leading order term of the
drift in $y$-direction averaged over the periodic motion of the linearization.

Bifurcations without parameters can also appear combined with additional parameters,
for example $g(y) = g(y_1,y_2) = (g_1(y_1,y_2),0)$. 
For Bogdanov-Takens points the case of a generic vector field 
with two-dimensional equilibrium manifold turns out to be equivalent to 
the case of a generic one-parameter family of vector fields with one-dimensional
equilibrium manifolds,
at least up to leading order of the suitably rescaled normal form; 
see \cite{FiedlerLiebscher01-TakensBogdanov}.
Both viewpoints are closely related in this case.
In the example of a transcritical point with drift singularity, 
studied in section \ref{Lie:secTranscritDriftSingularity}, 
however, both settings lead to drastically different dynamical systems, 
see Remark \ref{Lie:remTranscritDriftSingularityNoEquilibria}.

In the present paper we analyze the case $x\in\setR$, $y\in\setR^m$
of dynamical systems with a codimension-one manifold of equilibria.
As it turns out, the dynamics near bifurcation points of codimension $m$
on these manifolds can be related to singularities of smooth maps $h:\setR^{m+1}\to\setR$;
the Theorem \ref{Lie:thSingularity}.
This correspondence permits the application of singularity theory 
and most notably of the classification of singularities 
to bifurcations without parameters.
It might serve as a first step towards 
a classification of general bifurcations without parameters.

In section \ref{Lie:secTranscritDriftSingularity} 
we will start with an example to illustrate the general theorem formulated and proved 
in section \ref{Lie:secGeneralTheorem}.
In section \ref{Lie:secDiscussion} we conclude with a discussion 
of further questions, possibilities and problems.

\textbf{Acknowledgement:} This work was supported by the Deutsche Forschungsgemeinschaft.


\section{Transcritical points with drift singularity}
\label{Lie:secTranscritDriftSingularity}

Let us first discuss the cases $m=1,2$ of bifurcations 
along one- and two-dimensional equilibrium manifolds 
as an example to illustrate the general theorem in section \ref{Lie:secGeneralTheorem}.

\subsection{Transcritical point}

The case $m=1$ of a line of equilibria in a 2-dimensional center manifold 
has already been studied in \cite{Liebscher97-Diplom}, 
see also \cite{FiedlerLiebscher02-ICM-Beijing}.
In classical bifurcation theory, the only robust bifurcation of
\begin{equation}\label{Lie:eqClassicalTranscrit}
\begin{array}{rcll}
 \dot{x}       &=& f(x,\lambda) \quad& \in\setR, \qquad f(0,\lambda) \equiv 0,\\
 \dot{\lambda} &=& 0                 & \in\setR
\end{array}
\end{equation}
is the transcritical bifurcation with normal form
\begin{equation}\label{Lie:eqClassicalTranscritNF}
 \dot{x} \;=\; x(x-\lambda) + \hot
\end{equation}
see figure \ref{Lie:figTranscritical}(a).
It is caused by the nontrivial eigenvalue of the linearization at the equilibrium $x=0$
crossing zero with nonvanishing speed as the parameter $\lambda$ increases.
Together with the non-degeneracy condition $\partial_x^2f(0,0) \neq 0$, 
this implies the above normal form (\ref{Lie:eqClassicalTranscritNF}).

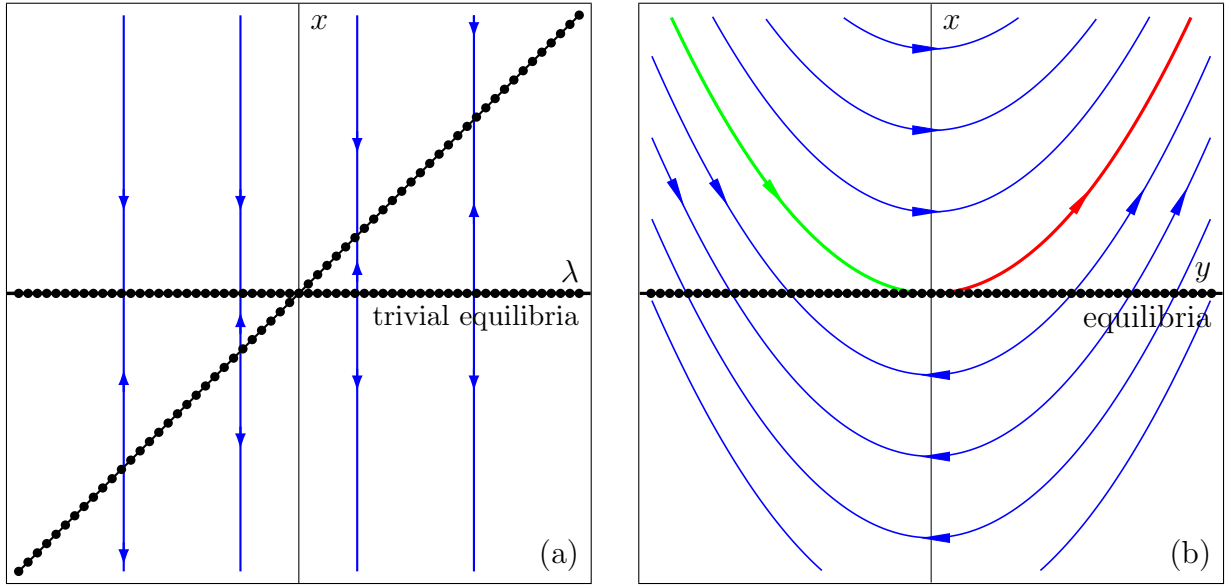
\begin{figure}
\setlength{\unitlength}{0.48\textwidth}
\begin{picture}(1.0,1.0)(0.0,0.0)
\put(0.20,0.02){\thicklines\color{blue}\line(0,1){0.96}}
\put(0.40,0.02){\thicklines\color{blue}\line(0,1){0.96}}
\put(0.60,0.02){\thicklines\color{blue}\line(0,1){0.96}}
\put(0.80,0.02){\thicklines\color{blue}\line(0,1){0.96}}
\put(0.20,0.07){\thicklines\color{blue}\vector(0,-1){0.04}}
\put(0.40,0.27){\thicklines\color{blue}\vector(0,-1){0.04}}
\put(0.60,0.37){\thicklines\color{blue}\vector(0,-1){0.04}}
\put(0.80,0.37){\thicklines\color{blue}\vector(0,-1){0.04}}
\put(0.20,0.68){\thicklines\color{blue}\vector(0,-1){0.04}}
\put(0.40,0.68){\thicklines\color{blue}\vector(0,-1){0.04}}
\put(0.60,0.78){\thicklines\color{blue}\vector(0,-1){0.04}}
\put(0.80,0.98){\thicklines\color{blue}\vector(0,-1){0.04}}
\put(0.20,0.33){\thicklines\color{blue}\vector(0,1){0.04}}
\put(0.40,0.43){\thicklines\color{blue}\vector(0,1){0.04}}
\put(0.60,0.52){\thicklines\color{blue}\vector(0,1){0.04}}
\put(0.80,0.62){\thicklines\color{blue}\vector(0,1){0.04}}
\put(0.98,0.52){\makebox(0,0)[rb]{$\lambda$}}
\put(0.98,0.48){\makebox(0,0)[rt]{\small trivial equilibria}}
\put(0.52,0.98){\makebox(0,0)[lt]{$x$}}
\put(0.00,0.50){\linethickness{1.3pt}\line(1,0){1}}
\multiput(0.02,0.50)(0.016,0.0){61}{\circle*{0.016}}
\put(0.50,0.00){\line(0,1){1}}
\put(0.02,0.02){\thicklines\line(1,1){0.96}}
\multiput(0.02,0.02)(0.016,0.016){61}{\circle*{0.016}}
\put(0,0){\framebox(1,1){}}
\put(0.98,0.02){\makebox(0,0)[br]{(a)}}
\end{picture}
\hfill
\begin{picture}(1.0,1.0)(0.0,0.0)
\put(-0.02,-0.02){\makebox(1.04,1.04){%
 \includegraphics[width=1.04\unitlength]{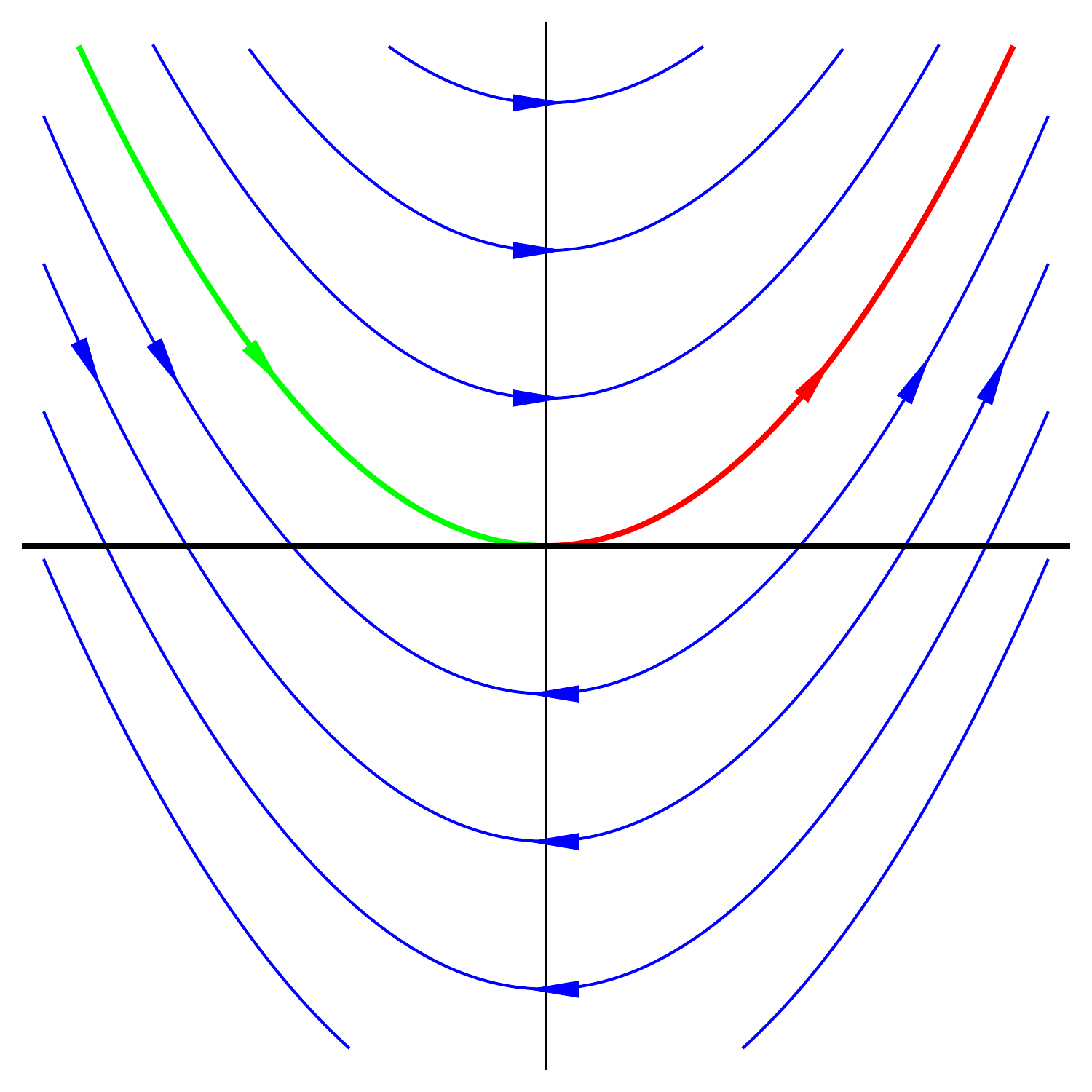}}}
\put(0.98,0.52){\makebox(0,0)[rb]{$y$}}
\put(0.98,0.48){\makebox(0,0)[rt]{equilibria}}
\put(0.52,0.98){\makebox(0,0)[lt]{$x$}}
\multiput(0.02,0.50)(0.016,0.0){61}{\circle*{0.016}}
\put(0,0){\framebox(1,1){}}
\put(0.98,0.02){\makebox(0,0)[br]{(b)}}
\end{picture}
\caption{\label{Lie:figTranscritical}
Transcritical bifurcation point: classical (a) and without parameters (b).}
\end{figure}

Without parameters, 
\begin{equation}\label{Lie:eqTranscrit}
\begin{array}{rcll}
 \dot{x}       &=& f(x,y) \quad& \in\setR, \qquad f(0,y) \equiv 0,\\
 \dot{y}       &=& g(x,y)      & \in\setR, \qquad g(0,y) \equiv 0,
\end{array}
\end{equation}
the nontrivial eigenvalue $\partial_xf(0,y)$ can change sign along lines of equilibria $\{y=0\}$.
The generic normal form reads
\begin{equation}\label{Lie:eqTranscritNF}
\begin{array}{rcll}
 \dot{x}       &=& xy + \hot,\\
 \dot{y}       &=& x,
\end{array}
\end{equation}
see figure \ref{Lie:figTranscritical}(b).
It requires the same transversality condition of the nontrivial eigenvalue 
as the classical transcritical bifurcation.
The non-degeneracy condition, however, is replaced by $\partial_x g(0,0) \neq 0$ 
and yields the two-dimensional Jordan block of the linearization at the transcritical point.
Trajectories form parabolas with tangency to the line of equilibria at the transcritical point.
The flow direction is reversed on opposite sides of the equilibrium line.

\subsection{Parameter dependent transcritical point with drift singularity}

Along two-dimensional equilibrium manifolds we expect transcritical points 
to form one-dimensional curves, by implicit function theorem.
At isolated points one of the non-degeneracy conditions may fail and codimension-two 
singularities appear.
We will discuss the case of failing drift condition, 
first in a one-parameter-family on lines of equilibria 
and then along a two-dimensional equilibrium surface.

With one parameter, the setting is as follows. We consider a system
\begin{equation}\label{Lie:eqSemiclassicalDriftsingularity}
\left.
\begin{array}{rclcl}
\left(\begin{array}{c} \dot{x} \\ \dot{y} \end{array}\right)
 &=&
 F(x,y,\lambda)
 &=&
 \left(\begin{array}{c} f(x,y,\lambda) \\ g(x,y,\lambda) \end{array}\right)
\\[3ex]
\dot{\lambda} \;\;\;\; &=& 0,
\end{array}
\qquad\right\}\quad
x,y,\lambda \in \setR,
\end{equation}
with the following properties:
\begin{enumerate}
\item 
For all parameter values, 
there exists a line of equilibria, $F(0,y,\lambda) \equiv 0$,
forming a plane of equilibria in the extended phase space.
\item
For all parameter values, 
the origin is a transcritical point, 
i.e. the origin has an eigenvalue zero in transverse direction to the equilibrium plane, 
$\partial_x f(0,0,\lambda) \equiv 0$.
\item
For all parameter values, 
this nontrivial eigenvalue crosses zero with nonvanishing speed as $y$ increases,
$\partial_y \partial_x f(0,0,0) > 0$.
\item
At $\lambda=0$ the drift non-degeneracy condition fails, $\partial_x g(0,0,0) = 0$.
\item
This drift degeneracy is transverse, 
i.e. the drift changes direction with nonvanishing speed, as $\lambda$ increases,
$\partial_\lambda \partial_x g(0,0,0) > 0$.
\end{enumerate}
The first condition is our structural assumption, 
(iii,v) are non-degeneracy assumptions fulfilled generically, 
and (ii,iv) describe our bifurcation point.
This setup is robust, i.e. under small perturbations of $F$ respecting (i) 
there is a point near the origin satisfying (ii--v) for the perturbed system.
From the viewpoint of singularity theory, (ii,iv) define a singularity of codimension two
that is unfolded versally by the coordinate $y$ along the line of trivial equilibria 
and the parameter $\lambda$.

Condition (i) allows us to factor out $x$,
\begin{equation}\label{Lie:eqSemiclassicalDriftsingularityFactorX}
F(x,y,\lambda) \;=\; x\tilde{F}(x,y,\lambda),
\end{equation}
with smooth $\tilde{F}$.
Conditions (ii-v) yield an expansion
\begin{equation}\label{Lie:eqSemiclassicalDriftsingularityExpansion}
\tilde{F}(x,y,\lambda) \;=\;
\left(\begin{array}{ll} ax + by \\ cx + dy + \sigma\lambda \end{array}\right)
 + \Ord((|x|+|y|+|\lambda|)^2),
\end{equation}
with coefficients $a,b,c,d,\sigma\in\setR$, $b > 0$, $\sigma > 0$.
We assume an additional non-degeneracy condition
\begin{enumerate}
\addtocounter{enumi}{5}
\item The matrix
\[
\partial_{(x,y)}\left( \frac{1}{x}F \right)(0,0,0) \;=\;
\left(\begin{array}{cc} a & b \\ c & d \end{array}\right)
\]
is hyperbolic, i.e. has no purely imaginary eigenvalues.
\end{enumerate}
Setting
\begin{equation}\label{Lie:eqSemiclassicalDriftsingularityDetTrace}
\delta \,:=\; ad-bc, \qquad \tau\,:=\;a+d,
\end{equation}
for determinant and trace, 
we therefore have $\delta \neq 0$, and $\tau\neq0$ if $\delta>0$.

Applying the multiplier $x^{-1}$ to system (\ref{Lie:eqSemiclassicalDriftsingularityFactorX})
preserves trajectories for $x \neq 0$ but reverses their direction for $x < 0$.
After the coordinate transformation 
$\tilde{x} = x$, $\tilde{y} = ax+by$, $\tilde{\lambda} = b\sigma\lambda$,
we obtain
\begin{equation}\label{Lie:eqSemiclassicalDriftsingularityNF}
\left(\begin{array}{c} \tilde{x}' \\ \tilde{y}' \end{array}\right)
\;=\;
\left(\begin{array}{rrr} & \tilde{y} & \\ 
      -\delta \tilde{x} & 
      \strut\!\!\!\!+ \tau \tilde{y} & 
      \strut\!\!\!\!+ \tilde{\lambda} \end{array}\right)
 + \Ord((|x|+|y|+|\lambda|)^2).
\end{equation}
This yields a bifurcating equilibrium 
at $(\tilde{x},\tilde{y}) \approx (\tilde{\lambda}/\delta,0)$. 
Transversality of the branch of equilibria with respect to the trivial line of equilibria
as well as the hyperbolicity of the nontrivial equilibria is ensured by condition (vi).
Therefore, terms of higher order in (\ref{Lie:eqSemiclassicalDriftsingularityNF})
will preserve this structure.
See figure \ref{Lie:figSemiclassicalDriftsingularity} for phase portraits in various cases.
Note the appearance of the generic transcritical bifurcation without parameters, 
figure \ref{Lie:figTranscritical}, for $\lambda\neq0$.

\begin{figure}
\setlength{\unitlength}{0.325\textwidth}
\centering
\begin{picture}(1.0,1.0)(-0.5,-0.5)
\put( 0.00, 0.20){%
\thicklines\color{blue}
\put(-0.08, 0.06){\vector(4,-3){0}}
\put( 0.00, 0.00){\line(-4,3){0.37}}
\put( 0.08,-0.06){\vector(-4,3){0}}
\put( 0.44,-0.33){\vector(4,-3){0}}
\put( 0.00, 0.00){\line(4,-3){0.48}}
\put(-0.15,-0.12){\vector(-5,-4){0}}
\put(-0.35,-0.28){\vector(5,4){0}}
\put( 0.00, 0.00){\line(-5,-4){0.48}}
\put( 0.15, 0.12){\vector(5,4){0}}
\put( 0.00, 0.00){\line(5,4){0.35}}
\qbezier(-0.1875,-0.19)(0,-0.04)(0.2,-0.19)
\put(-0.1875,-0.19){\line(-5,-4){0.2925}}
\put( 0.2, -0.19){\line(4,-3){0.28}}
\put( 0.00,-0.305){\vector(1,0){0}}
\qbezier(-0.48,-0.504 )( 0,-0.12)(0.48, -0.48)
\put( 0.00,-0.114){\vector(-1,0){0}}
\color{green}
\put(-0.48,-0.464){\line(5,4){0.13}}
\put(-0.35,-0.36){\vector(5,4){0}}
\qbezier(-0.35,-0.36)(-0.15,-0.2)(0.0,-0.2)
\color{red}
\qbezier(0.0,-0.2)(0.16,-0.2)(0.36,-0.35)
\put( 0.36,-0.35){\line(4,-3){0.12}}
\put( 0.36,-0.35){\vector(4,-3){0}}
\color{black}\thinlines
\put( 0.00, 0.00){\circle*{0.02}}
}
\put( 0.00,-0.50){\line(0,1){1}}
\put( 0.02, 0.48){\makebox(0,0)[lt]{$\tilde{x}$}}
\multiput(-0.488,0.0)(0.016,0.0){62}{\circle*{0.016}}
\put( 0.48, 0.02){\makebox(0,0)[rb]{$\tilde{y}$}}
\put( 0.48,-0.02){\makebox(0,0)[rt]{\scriptsize equilibria}}
\put(-0.50,-0.50){\framebox(1,1){}}
\put(-0.02,-0.48){\makebox(0,0)[br]{$\tilde{\lambda} < 0$}}
\end{picture}
\hfill
\begin{picture}(1.0,1.0)(-0.5,-0.5)
\put( 0.00, 0.00){%
\thicklines\color{green}
\put(-0.20, 0.15){\vector(4,-3){0}}
\put( 0.00, 0.00){\line(-4,3){0.48}}
\put(-0.20,-0.16){\vector(5,4){0}}
\put( 0.00, 0.00){\line(-5,-4){0.48}}
\color{red}
\put( 0.24,-0.18){\vector(4,-3){0}}
\put( 0.00, 0.00){\line(4,-3){0.48}}
\put( 0.24, 0.192){\vector(5,4){0}}
\put( 0.00, 0.00){\line(5,4){0.48}}
\color{blue}
\put( 0.00,-0.285){\vector(1,0){0}}
\qbezier(-0.475,-0.48 )( 0,-0.1)(0.48, -0.46)
\put( 0.00, 0.285){\vector(1,0){0}}
\qbezier(-0.48,  0.46 )( 0, 0.1)(0.475, 0.48)
\put( 0.34, 0.11){\vector(1,3){0}}
\put( 0.35,-0.12){\vector(2,-5){0}}
\qbezier(-0.48, -0.264)(-0.15,0)(-0.48, 0.2475)
\put(-0.33,-0.08){\vector(1,3){0}}
\put(-0.33, 0.07){\vector(2,-5){0}}
\qbezier( 0.48,  0.264)( 0.15,0)( 0.48,-0.2475)
\color{black}\thinlines
\put( 0.00, 0.00){\circle*{0.02}}
}
\put( 0.00,-0.50){\line(0,1){1}}
\put( 0.02, 0.48){\makebox(0,0)[lt]{$\tilde{x}$}}
\multiput(-0.488,0.0)(0.016,0.0){62}{\circle*{0.016}}
\put( 0.48, 0.02){\makebox(0,0)[rb]{$\tilde{y}$}}
\put( 0.48,-0.02){\makebox(0,0)[rt]{\scriptsize equilibria}}
\put(-0.50,-0.50){\framebox(1,1){}}
\put(-0.02,-0.48){\makebox(0,0)[br]{$\tilde{\lambda} = 0$}}
\end{picture}
\hfill
\begin{picture}(1.0,1.0)(-0.5,-0.5)
\put( 0.00, -0.20){%
\thicklines\color{blue}
\put(-0.08,-0.06){\vector(4,3){0}}
\put( 0.00, 0.00){\line(-4,-3){0.37}}
\put( 0.08, 0.06){\vector(-4,-3){0}}
\put( 0.44, 0.33){\vector(4,3){0}}
\put( 0.00, 0.00){\line(4,3){0.48}}
\put(-0.15, 0.12){\vector(-5,4){0}}
\put(-0.35, 0.28){\vector(5,-4){0}}
\put( 0.00, 0.00){\line(-5,4){0.48}}
\put( 0.15,-0.12){\vector(5,-4){0}}
\put( 0.00, 0.00){\line(5,-4){0.35}}
\qbezier(-0.1875,0.19)(0,0.04)(0.2,0.19)
\put(-0.1875,0.19){\line(-5,4){0.2925}}
\put( 0.2,  0.19){\line(4,3){0.28}}
\put( 0.00, 0.305){\vector(1,0){0}}
\qbezier(-0.48,0.504)( 0,0.12)(0.48,0.48)
\put( 0.00, 0.114){\vector(-1,0){0}}
\color{green}
\put(-0.48, 0.464){\line(5,-4){0.13}}
\put(-0.35, 0.36){\vector(5,-4){0}}
\qbezier(-0.35,0.36)(-0.15,0.2)(0.0,0.2)
\color{red}
\qbezier(0.0,0.2)(0.16,0.2)(0.36,0.35)
\put( 0.36, 0.35){\line(4,3){0.12}}
\put( 0.36, 0.35){\vector(4,3){0}}
\color{black}\thinlines
\put( 0.00, 0.00){\circle*{0.02}}
}
\put( 0.00,-0.50){\line(0,1){1}}
\put( 0.02, 0.48){\makebox(0,0)[lt]{$\tilde{x}$}}
\multiput(-0.488,0.0)(0.016,0.0){62}{\circle*{0.016}}
\put( 0.48, 0.02){\makebox(0,0)[rb]{$\tilde{y}$}}
\put( 0.48,-0.02){\makebox(0,0)[rt]{\scriptsize equilibria}}
\put(-0.50,-0.50){\framebox(1,1){}}
\put(-0.02,-0.48){\makebox(0,0)[br]{$\tilde{\lambda} > 0$}}
\end{picture}
bifurcating saddle

\bigskip
\bigskip

\begin{picture}(1.0,1.0)(-0.5,-0.5)
\put( 0.00,-0.17){%
\thicklines\color{blue}
{\setlength{\unitlength}{2\unitlength}
\put( 0.20, 0.00 ){\vector(1,-4){0}}
\qbezier( 0.169,-0.154)( 0.218,-0.090)( 0.20 , 0.00 )}%
\put(-0.02, 0.338){\vector(-1,0){0}}
\qbezier( 0.40 , 0.00 )( 0.366, 0.166)( 0.233, 0.259)
\qbezier( 0.233, 0.259)( 0.117, 0.351)(-0.034, 0.336)
\qbezier(-0.034, 0.336)(-0.171, 0.322)(-0.240, 0.218)
\color{red}
\put(-0.283, 0.00 ){\vector(-1,4){0}}
\put( 0.20, 0.00 ){\vector(1,-4){0}}
\qbezier(-0.240, 0.218)(-0.308, 0.128)(-0.283, 0.00 )
\qbezier(-0.283, 0.00 )(-0.259,-0.117)(-0.165,-0.183)
\qbezier(-0.165,-0.183)(-0.082,-0.248)( 0.024,-0.238)
\qbezier( 0.024,-0.238)( 0.122,-0.228)( 0.169,-0.154)
\qbezier( 0.169,-0.154)( 0.218,-0.090)( 0.20 , 0.00 )
\color{blue}
\put(-0.240, 0.218){\line(-6,-10){0.03}}
\color{green}
{\setlength{\unitlength}{0.5\unitlength}\color{red}
\qbezier( 0.40 , 0.00 )( 0.366, 0.166)( 0.233, 0.259)
\qbezier( 0.233, 0.259)( 0.117, 0.351)(-0.034, 0.336)
\color{green}
\put(-0.283, 0.00 ){\vector(-1,3){0}}
\put( 0.20 , 0.00 ){\vector(1,-3){0}}
\qbezier(-0.034, 0.336)(-0.171, 0.322)(-0.240, 0.218)
\qbezier(-0.240, 0.218)(-0.308, 0.128)(-0.283, 0.00 )
\qbezier(-0.283, 0.00 )(-0.259,-0.117)(-0.165,-0.183)
\qbezier(-0.165,-0.183)(-0.082,-0.248)( 0.024,-0.238)
\qbezier( 0.024,-0.238)( 0.122,-0.228)( 0.169,-0.154)
\qbezier( 0.169,-0.154)( 0.218,-0.090)( 0.20 , 0.00 )}%
{\setlength{\unitlength}{0.25\unitlength}
\qbezier( 0.40 , 0.00 )( 0.366, 0.166)( 0.233, 0.259)
\qbezier( 0.233, 0.259)( 0.117, 0.351)(-0.034, 0.336)
\qbezier(-0.034, 0.336)(-0.171, 0.322)(-0.240, 0.218)
\qbezier(-0.240, 0.218)(-0.308, 0.128)(-0.283, 0.00 )
\qbezier(-0.283, 0.00 )(-0.259,-0.117)(-0.165,-0.183)
\qbezier(-0.165,-0.183)(-0.082,-0.248)( 0.024,-0.238)
\qbezier( 0.024,-0.238)( 0.122,-0.228)( 0.169,-0.154)
\qbezier( 0.169,-0.154)( 0.218,-0.090)( 0.20 , 0.00 )}%
{\setlength{\unitlength}{0.125\unitlength}
\qbezier( 0.40 , 0.00 )( 0.366, 0.166)( 0.233, 0.259)
\qbezier( 0.233, 0.259)( 0.117, 0.351)(-0.034, 0.336)
\qbezier(-0.034, 0.336)(-0.171, 0.322)(-0.240, 0.218)
\qbezier(-0.240, 0.218)(-0.308, 0.128)(-0.283, 0.00 )
\qbezier(-0.283, 0.00 )(-0.259,-0.117)(-0.165,-0.183)
\qbezier(-0.165,-0.183)(-0.082,-0.248)( 0.024,-0.238)
\qbezier( 0.024,-0.238)( 0.122,-0.228)( 0.169,-0.154)
\qbezier( 0.169,-0.154)( 0.218,-0.090)( 0.20 , 0.00 )}%
\color{black}\thinlines
\put( 0.00, 0.00){\circle*{0.02}}
}
\put( 0.00,-0.50){\line(0,1){1}}
\put( 0.02, 0.48){\makebox(0,0)[lt]{$\tilde{x}$}}
\multiput(-0.488,0.0)(0.016,0.0){62}{\circle*{0.016}}
\put( 0.48, 0.02){\makebox(0,0)[rb]{$\tilde{y}$}}
\put( 0.48,-0.02){\makebox(0,0)[rt]{\scriptsize equilibria}}
\put(-0.50,-0.50){\framebox(1,1){}}
\put(-0.48,-0.48){\makebox(0,0)[bl]{$\tilde{\lambda} < 0$}}
\end{picture}
\hfill
\begin{picture}(1.0,1.0)(-0.5,-0.5)
\put( 0.00, 0.00){%
\thicklines\color{blue}
{\setlength{\unitlength}{2\unitlength}
\put(-0.01,-0.238){\vector(-1,0){0}}
\qbezier(-0.165,-0.183)(-0.082,-0.248)( 0.024,-0.238)
\qbezier( 0.024,-0.238)( 0.122,-0.228)( 0.169,-0.154)
\qbezier( 0.169,-0.154)( 0.218,-0.090)( 0.20 , 0.00 )}%
\put(-0.02, 0.338){\vector(-1,0){0}}
\put(-0.02,-0.238){\vector(-1,0){0}}
\qbezier( 0.40 , 0.00 )( 0.366, 0.166)( 0.233, 0.259)
\qbezier( 0.233, 0.259)( 0.117, 0.351)(-0.034, 0.336)
\qbezier(-0.034, 0.336)(-0.171, 0.322)(-0.240, 0.218)
\qbezier(-0.240, 0.218)(-0.308, 0.128)(-0.283, 0.00 )
\qbezier(-0.283, 0.00 )(-0.259,-0.117)(-0.165,-0.183)
\qbezier(-0.165,-0.183)(-0.082,-0.248)( 0.024,-0.238)
\qbezier( 0.024,-0.238)( 0.122,-0.228)( 0.169,-0.154)
\qbezier( 0.169,-0.154)( 0.218,-0.090)( 0.20 , 0.00 )
{\setlength{\unitlength}{0.5\unitlength}
\put(-0.04, 0.338){\vector(-1,0){0}}
\put(-0.04,-0.238){\vector(-1,0){0}}
\qbezier( 0.40 , 0.00 )( 0.366, 0.166)( 0.233, 0.259)
\qbezier( 0.233, 0.259)( 0.117, 0.351)(-0.034, 0.336)
\qbezier(-0.034, 0.336)(-0.171, 0.322)(-0.240, 0.218)
\qbezier(-0.240, 0.218)(-0.308, 0.128)(-0.283, 0.00 )
\qbezier(-0.283, 0.00 )(-0.259,-0.117)(-0.165,-0.183)
\qbezier(-0.165,-0.183)(-0.082,-0.248)( 0.024,-0.238)
\qbezier( 0.024,-0.238)( 0.122,-0.228)( 0.169,-0.154)
\qbezier( 0.169,-0.154)( 0.218,-0.090)( 0.20 , 0.00 )}%
{\setlength{\unitlength}{0.25\unitlength}
\put(-0.08, 0.338){\vector(-1,0){0}}
\put(-0.08,-0.238){\vector(-1,0){0}}
\qbezier( 0.40 , 0.00 )( 0.366, 0.166)( 0.233, 0.259)
\qbezier( 0.233, 0.259)( 0.117, 0.351)(-0.034, 0.336)
\qbezier(-0.034, 0.336)(-0.171, 0.322)(-0.240, 0.218)
\qbezier(-0.240, 0.218)(-0.308, 0.128)(-0.283, 0.00 )
\qbezier(-0.283, 0.00 )(-0.259,-0.117)(-0.165,-0.183)
\qbezier(-0.165,-0.183)(-0.082,-0.248)( 0.024,-0.238)
\qbezier( 0.024,-0.238)( 0.122,-0.228)( 0.169,-0.154)
\qbezier( 0.169,-0.154)( 0.218,-0.090)( 0.20 , 0.00 )}%
{\setlength{\unitlength}{0.125\unitlength}
\qbezier( 0.40 , 0.00 )( 0.366, 0.166)( 0.233, 0.259)
\qbezier( 0.233, 0.259)( 0.117, 0.351)(-0.034, 0.336)
\qbezier(-0.034, 0.336)(-0.171, 0.322)(-0.240, 0.218)
\qbezier(-0.240, 0.218)(-0.308, 0.128)(-0.283, 0.00 )
\qbezier(-0.283, 0.00 )(-0.259,-0.117)(-0.165,-0.183)
\qbezier(-0.165,-0.183)(-0.082,-0.248)( 0.024,-0.238)
\qbezier( 0.024,-0.238)( 0.122,-0.228)( 0.169,-0.154)
\qbezier( 0.169,-0.154)( 0.218,-0.090)( 0.20 , 0.00 )}%
\color{black}\thinlines
\put( 0.00, 0.00){\circle*{0.02}}
}
\put( 0.00,-0.50){\line(0,1){1}}
\put( 0.02, 0.48){\makebox(0,0)[lt]{$\tilde{x}$}}
\multiput(-0.488,0.0)(0.016,0.0){62}{\circle*{0.016}}
\put( 0.48, 0.02){\makebox(0,0)[rb]{$\tilde{y}$}}
\put( 0.48,-0.02){\makebox(0,0)[rt]{\scriptsize equilibria}}
\put(-0.50,-0.50){\framebox(1,1){}}
\put(-0.48,-0.48){\makebox(0,0)[bl]{$\tilde{\lambda} = 0$}}
\end{picture}
\hfill
\begin{picture}(1.0,1.0)(-0.5,-0.5)
\put( 0.00, 0.12){%
\thicklines\color{blue}
{\setlength{\unitlength}{2\unitlength}
\put( 0.20, 0.00 ){\vector(-1,6){0}}
\put(-0.01,-0.238){\vector(-1,0){0}}
\qbezier(-0.165,-0.183)(-0.082,-0.248)( 0.024,-0.238)
\qbezier( 0.024,-0.238)( 0.122,-0.228)( 0.169,-0.154)
\qbezier( 0.169,-0.154)( 0.218,-0.090)( 0.20 , 0.00 )}%
\put(-0.283,0.00 ){\vector(1,-6){0}}
\put(-0.02,-0.238){\vector(-1,0){0}}
\qbezier( 0.40 , 0.00 )( 0.366, 0.166)( 0.233, 0.259)
\qbezier( 0.233, 0.259)( 0.117, 0.351)(-0.034, 0.336)
\qbezier(-0.034, 0.336)(-0.171, 0.322)(-0.240, 0.218)
\qbezier(-0.240, 0.218)(-0.308, 0.128)(-0.283, 0.00 )
\qbezier(-0.283, 0.00 )(-0.259,-0.117)(-0.165,-0.183)
\qbezier(-0.165,-0.183)(-0.082,-0.248)( 0.024,-0.238)
\qbezier( 0.024,-0.238)( 0.122,-0.228)( 0.169,-0.154)
\color{green}
\put( 0.20, 0.00 ){\vector(-1,8){0}}
\qbezier( 0.169,-0.154)( 0.218,-0.090)( 0.20 , 0.00 )
\color{blue}
\put( 0.169,-0.154){\line(4,7){0.02}}
\color{red}
{\setlength{\unitlength}{0.5\unitlength}\color{green}
\put(-0.283,0.00 ){\vector(1,-8){0}}
\qbezier( 0.40 , 0.00 )( 0.366, 0.166)( 0.233, 0.259)
\qbezier( 0.233, 0.259)( 0.117, 0.351)(-0.034, 0.336)
\qbezier(-0.034, 0.336)(-0.171, 0.322)(-0.240, 0.218)
\qbezier(-0.240, 0.218)(-0.308, 0.128)(-0.283, 0.00 )
\qbezier(-0.283, 0.00 )(-0.259,-0.117)(-0.165,-0.183)
\qbezier(-0.165,-0.183)(-0.082,-0.248)( 0.024,-0.238)
\color{red}
\put( 0.20, 0.00 ){\vector(0,1){0}}
\qbezier( 0.024,-0.238)( 0.122,-0.228)( 0.169,-0.154)
\qbezier( 0.169,-0.154)( 0.218,-0.090)( 0.20 , 0.00 )}%
{\setlength{\unitlength}{0.25\unitlength}
\put(-0.29,0.00 ){\vector(-1,-8){0}}
\qbezier( 0.40 , 0.00 )( 0.366, 0.166)( 0.233, 0.259)
\qbezier( 0.233, 0.259)( 0.117, 0.351)(-0.034, 0.336)
\qbezier(-0.034, 0.336)(-0.171, 0.322)(-0.240, 0.218)
\qbezier(-0.240, 0.218)(-0.308, 0.128)(-0.283, 0.00 )
\qbezier(-0.283, 0.00 )(-0.259,-0.117)(-0.165,-0.183)
\qbezier(-0.165,-0.183)(-0.082,-0.248)( 0.024,-0.238)
\qbezier( 0.024,-0.238)( 0.122,-0.228)( 0.169,-0.154)
\qbezier( 0.169,-0.154)( 0.218,-0.090)( 0.20 , 0.00 )}%
{\setlength{\unitlength}{0.125\unitlength}
\qbezier( 0.40 , 0.00 )( 0.366, 0.166)( 0.233, 0.259)
\qbezier( 0.233, 0.259)( 0.117, 0.351)(-0.034, 0.336)
\qbezier(-0.034, 0.336)(-0.171, 0.322)(-0.240, 0.218)
\qbezier(-0.240, 0.218)(-0.308, 0.128)(-0.283, 0.00 )
\qbezier(-0.283, 0.00 )(-0.259,-0.117)(-0.165,-0.183)
\qbezier(-0.165,-0.183)(-0.082,-0.248)( 0.024,-0.238)
\qbezier( 0.024,-0.238)( 0.122,-0.228)( 0.169,-0.154)
\qbezier( 0.169,-0.154)( 0.218,-0.090)( 0.20 , 0.00 )}%
\color{black}\thinlines
\put( 0.00, 0.00){\circle*{0.02}}
}
\put( 0.00,-0.50){\line(0,1){1}}
\put( 0.02, 0.48){\makebox(0,0)[lt]{$\tilde{x}$}}
\multiput(-0.488,0.0)(0.016,0.0){62}{\circle*{0.016}}
\put( 0.48, 0.02){\makebox(0,0)[rb]{$\tilde{y}$}}
\put( 0.48,-0.02){\makebox(0,0)[rt]{\scriptsize equilibria}}
\put(-0.50,-0.50){\framebox(1,1){}}
\put(-0.48,-0.48){\makebox(0,0)[bl]{$\tilde{\lambda} > 0$}}
\end{picture}
bifurcating focus

\bigskip
\bigskip

\begin{picture}(1.0,1.0)(-0.5,-0.5)
\put( 0.00,-0.20){%
\thicklines\color{blue}
\qbezier( 0.00, 0.32)( 0.16, 0.32)( 0.48, 0.08)
\qbezier( 0.00, 0.32)(-0.24, 0.32)( 0.00, 0.00)
\qbezier( 0.00, 0.08)( 0.08, 0.08)( 0.48,-0.22)
\qbezier( 0.00, 0.08)(-0.06, 0.08)( 0.00, 0.00)
\put( 0.00, 0.32){\vector(-1,0){0}}
\put( 0.44, 0.11){\vector(4,-3){0}}
\put( 0.24,-0.05){\vector(4,-3){0}}
\color{red}
\qbezier( 0.00, 0.195)( 0.12, 0.195)( 0.48,-0.07)
\put( 0.34, 0.03){\vector(4,-3){0}}
\color{green}
\qbezier( 0.00, 0.195)(-0.15, 0.195)( 0.00, 0.005)
\color{blue}
\put(0,0){\line(4,-3){0.37}}
\put(0,0){\line(3,-4){0.21}}
\put( 0.17,-0.1275){\vector(4,-3){0}}
\put( 0.15,-0.1115){\vector(4,-3){0}}
\put( 0.12,-0.16  ){\vector(3,-4){0}}
\put(0,0){\line(-4,3){0.48}}
\put(0,0){\line(-3,4){0.48}}
\put(-0.17, 0.1275){\vector(-4,3){0}}
\put(-0.15, 0.1115){\vector(-4,3){0}}
\put(-0.37, 0.28  ){\vector(4,-3){0}}
\put(-0.12, 0.16  ){\vector(-3,4){0}}
\put(-0.18, 0.24  ){\vector(3,-4){0}}
\color{black}\thinlines
\put( 0.00, 0.00){\circle*{0.02}}
}
\put( 0.00,-0.50){\line(0,1){1}}
\put( 0.02, 0.48){\makebox(0,0)[lt]{$\tilde{x}$}}
\multiput(-0.488,0.0)(0.016,0.0){62}{\circle*{0.016}}
\put( 0.48, 0.02){\makebox(0,0)[rb]{$\tilde{y}$}}
\put( 0.48,-0.02){\makebox(0,0)[rt]{\scriptsize equilibria}}
\put(-0.50,-0.50){\framebox(1,1){}}
\put(-0.48,-0.48){\makebox(0,0)[bl]{$\tilde{\lambda} < 0$}}
\end{picture}
\hfill
\begin{picture}(1.0,1.0)(-0.5,-0.5)
\put( 0.00, 0.00){%
\thicklines\color{red}
\put(0,0){\line(4,-3){0.48}}
\put(0,0){\line(3,-4){0.36}}
\put( 0.29,-0.2175){\vector(4,-3){0}}
\put( 0.27,-0.2015){\vector(4,-3){0}}
\put( 0.21,-0.28  ){\vector(3,-4){0}}
\qbezier(-0.32,0)(0,-0.24)(0.06,-0.16)
\qbezier(0,0)(0.09,-0.12)(0.06,-0.16)
\put(-0.24,-0.057){\vector(-4,3){0}}
\color{green}
\put(0,0){\line(-4,3){0.48}}
\put(0,0){\line(-3,4){0.36}}
\put(-0.25, 0.1875){\vector(4,-3){0}}
\put(-0.23, 0.1715){\vector(4,-3){0}}
\put(-0.18, 0.24  ){\vector(3,-4){0}}
\qbezier(0.32,0)(0,0.24)(-0.06,0.16)
\qbezier(0,0)(-0.09,0.12)(-0.06,0.16)
\put( 0.20, 0.087){\vector(-4,3){0}}
\color{blue}
\put(-0.32, 0   ){\line(-4,3){0.16}}
\put(-0.40, 0.06){\vector(4,-3){0}}
\put( 0.32, 0   ){\line(4,-3){0.16}}
\put( 0.44,-0.09){\vector(4,-3){0}}
\color{black}\thinlines
\put( 0.00, 0.00){\circle*{0.02}}
}
\put( 0.00,-0.50){\line(0,1){1}}
\put( 0.02, 0.48){\makebox(0,0)[lt]{$\tilde{x}$}}
\multiput(-0.488,0.0)(0.016,0.0){62}{\circle*{0.016}}
\put( 0.48, 0.02){\makebox(0,0)[rb]{$\tilde{y}$}}
\put( 0.48,-0.02){\makebox(0,0)[rt]{\scriptsize equilibria}}
\put(-0.50,-0.50){\framebox(1,1){}}
\put(-0.48,-0.48){\makebox(0,0)[bl]{$\tilde{\lambda} = 0$}}
\end{picture}
\hfill
\begin{picture}(1.0,1.0)(-0.5,-0.5)
\put( 0.00, 0.20){%
\thicklines\color{blue}
\qbezier( 0.00,-0.32)(-0.16,-0.32)(-0.48,-0.08)
\qbezier( 0.00,-0.32)( 0.24,-0.32)( 0.00, 0.00)
\qbezier( 0.00,-0.08)(-0.08,-0.08)(-0.48, 0.22)
\qbezier( 0.00,-0.08)( 0.06,-0.08)( 0.00, 0.00)
\put( 0.00,-0.32){\vector(-1,0){0}}
\put(-0.36,-0.17){\vector(4,-3){0}}
\put(-0.20, 0.02){\vector(4,-3){0}}
\color{green}
\qbezier( 0.00,-0.195)(-0.12,-0.195)(-0.48, 0.07)
\put(-0.28,-0.07){\vector(4,-3){0}}
\color{red}
\qbezier( 0.00,-0.195)( 0.15,-0.195)( 0.00,-0.005)
\color{blue}
\put(0,0){\line(-4,3){0.37}}
\put(0,0){\line(-3,4){0.21}}
\put(-0.13, 0.0975){\vector(4,-3){0}}
\put(-0.11, 0.0815){\vector(4,-3){0}}
\put(-0.09, 0.12  ){\vector(3,-4){0}}
\put(0,0){\line(4,-3){0.48}}
\put(0,0){\line(3,-4){0.48}}
\put( 0.13,-0.0975){\vector(-4,3){0}}
\put( 0.11,-0.0815){\vector(-4,3){0}}
\put( 0.41,-0.31  ){\vector(4,-3){0}}
\put( 0.09,-0.12  ){\vector(-3,4){0}}
\put( 0.24,-0.32  ){\vector(3,-4){0}}
\color{black}\thinlines
\put( 0.00, 0.00){\circle*{0.02}}
}
\put( 0.00,-0.50){\line(0,1){1}}
\put( 0.02, 0.48){\makebox(0,0)[lt]{$\tilde{x}$}}
\multiput(-0.488,0.0)(0.016,0.0){62}{\circle*{0.016}}
\put( 0.48, 0.02){\makebox(0,0)[rb]{$\tilde{y}$}}
\put( 0.48,-0.02){\makebox(0,0)[rt]{\scriptsize equilibria}}
\put(-0.50,-0.50){\framebox(1,1){}}
\put(-0.48,-0.48){\makebox(0,0)[bl]{$\tilde{\lambda} > 0$}}
\end{picture}
bifurcating node

\bigskip

Stable set of the origin in green, unstable set in red.

\caption{\label{Lie:figSemiclassicalDriftsingularity}
Drift singularity along a one-parameter family of transcritical points.}
\end{figure}

\subsection{Transcritical point with drift singularity without parameters}

Replacing the parameter $\lambda$ discussed above
by an additional direction of a plane of equilibria,
the drift along this manifold of equilibria is now a two-dimensional vector.
If will not vanish along generic one-dimensional curves. 
The drift singularity along curves of transcritical points 
is therefore not characterized by a vanishing drift 
but rather by a drift direction orthogonal to the curve of transcritical points.
(A drift in $\lambda$-direction was not possible before.)

The correct setup is given by a system
\begin{equation}\label{Lie:eqDriftsingularity}
\left(\begin{array}{c} \dot{x} \\ \dot{y} \end{array}\right)
 \;=\;
 F(x,y)
 \;=\;
 \left(\begin{array}{c} f(x,y) \\ g(x,y) \end{array}\right),
\qquad
x \in \setR,\quad y \in \setR^2,
\end{equation}
$y=(y_1,y_2)$, $g=(g_1,g_2)$, with the following properties:
\begin{enumerate}
\item 
The $y$-plane consists of equilibria, $F(0,y) \equiv 0$.
\item
There is a transcritical point at the origin, 
i.e. the $y$-plane loses normal hyperbolicity at this point,
$\partial_x f(0,0,0) = 0$.
\item
This loss of normal hyperbolicity 
is caused by the transverse eigenvalue crossing zero transversally,
$\nabla_{y} \partial_x f(0,0,0) \neq 0$.
Without loss of generality, the gradient points in $y_1$-direction, i.e.
$\partial_{y_1} \partial_x f(0,0,0) > 0$, $\partial_{y_2} \partial_x f(0,0,0) = 0$.
By implicit function theorem, this gives rise to a curve of transcritical points 
tangential to the $y_2$-axis.
\item
At the origin, the drift non-degeneracy transverse to the curve of 
transcritical points fails,
$\partial_x g_1(0,0,0) = 0$.
\item
This drift degeneracy is transverse, 
i.e. the drift direction crosses the tangent to the curve of transcritical points 
with nonvanishing speed along the curve of transcritical points,
$\partial_{y_1} \partial_x f(0,0,0) \; \partial_{y_2} \partial_x g_1(0,0,0) 
 + \partial_{y_2}^2 \partial_x f(0,0,0) \; \partial_x g_2(0,0,0) \neq 0$.
\item
The drift does not vanish at the origin, i.e. there is a component tangential 
to the curve of transcritical points,
$\partial_x g_2(0,0,0) > 0$.
\end{enumerate}
Note that conditions (i--v) correspond to the conditions of the previous section. 
Again, the degeneracies (ii,iv) are robust under perturbations satisfying (i), 
provided the non-degeneracy conditions (iii,v,vi) hold.
The signs of $\partial_{y_1} \partial_x f(0,0,0)$ and $\partial_x g(0,0,0)$ in (iii,v)
can be inverted by reflecting $y_1 \mapsto -y_1$ and $y_2 \mapsto -y_2$, respectively.

The non-degeneracy condition (vi) indeed yields
\begin{equation}\label{Lie:eqDriftNondegeneracy}
\left. 
  \frac{\diff}{\diff y_2} 
  \left\langle 
    \nabla_{(y_1,y_2)} \partial_x f, \partial_x g
  \right\rangle (0,\vartheta(y_2),y_2) 
\right|_{y_2=0} \neq 0,
\end{equation}
where $(x,y_1,y_2) = (0,\vartheta(y_2),y_2)$, $\vartheta(0) = 0$, $\vartheta'(0) = 0$,
is the curve $\gamma$ of transcritical points.
Locally, we could reparametrize $y$ to achieve $\vartheta \equiv 0$.
Conditions (ii,iii,vi) would then read:
$\partial_x f(0,0,y_2) \equiv 0$, 
$\partial_{y_1} \partial_x f(0,0,y_2) > 0$,
$\partial_{y_2} \partial_x g_1(0,0,0) \neq 0$.
But let us continue with the general setup.

As in the parameter-dependent case (\ref{Lie:eqSemiclassicalDriftsingularityFactorX}), 
we can factor out $x$ due to condition (i),
\begin{equation}\label{Lie:eqDriftsingularityFactorX}
F(x,y) \;=\; x\tilde{F}(x,y) 
       \;=\; x\left(\begin{array}{c} \tilde{f}(x,y) \\ \tilde{g}(x,y) \end{array}\right).
\end{equation}
However, this time, due to non-degeneracy (vi) no equilibrium remains,
\begin{equation}\label{Lie:eqDriftsingularityRegularOrigin}
\tilde{F}(0,0,0) \;=\; (0,0,\partial_x g_2(0,0,0)) \;\neq\; 0.
\end{equation}
Applying the flow-box theorem, there exists a local smooth diffeomorphism
\begin{equation}\label{Lie:eqDriftsingularityFlowboxTransform}
h(z_0,z_1,z_2) \;=\; \tilde{\Phi}_{z_2}(z_0,z_1,0),
\end{equation}
where $\tilde{\Phi}_t$ denotes the flow generated by the vector field $\tilde{F}$.
This diffeomorphism fixes the origin and transforms $\tilde{F}$ 
into the constant vectorfield,
\begin{equation}\label{Lie:eqDriftsingularityFlowbox}
[\Diff h(z_0,z_1,z_2)]^{-1} \tilde{F}(h(z_0,z_1,z_2))
\;=\; \left(\begin{array}{c} 0 \\ 0 \\ 1 \end{array}\right).
\end{equation}
Applying the same transformation to the original vector field $F$, we obtain
\begin{equation}\label{Lie:eqDriftsingularityTranform}
[\Diff h(z)]^{-1} F(h(z))
\;=\; [\Diff h(z)]^{-1} h_0(z) \tilde{F}(h(z))
\;=\; \left(\begin{array}{c} 0 \\ 0 \\ h_0(z) \end{array}\right),
\end{equation}
where $h=(h_0,h_1,h_2)$.

In a suitable neighborhood of the origin, the vectorfield $F$ is flow-equivalent to
a vectorfield
\begin{equation}\label{Lie:eqDriftsingularity1dVectorfield}
 \dot{z_2} \;=\; h_0(z_0,z_1,z_2)
\end{equation}
on the real line depending on two (classical) parameters $(z_0,z_1)$.
Expansion of $h_0$ using (\ref{Lie:eqDriftsingularityFlowboxTransform}) 
and conditions (ii-vi) yields
\begin{equation}\label{Lie:eqDriftsingularity1dExpansion}
 \dot{z_2} \;=\; a z_2^3 
  + (c_0 z_0 + c_1 z_1) z_2^2 
  + (b z_1 + c_2 z_0 + c_3 z_0^2 + c_4 z_0z_1 + c_5 z_1^2) z_2 
  + z_0 + \Ord(|z|^4)
\end{equation}
with
\begin{equation}\label{Lie:eqDriftsingularity1dCoeffs}
\begin{array}{rcl}
a &=& \left( 
        \partial_{y_1} \partial_x f(0) \; \partial_{y_2} \partial_x g_1(0) 
        + \partial_{y_2}^2 \partial_x f(0) \; \partial_x g_2(0) 
      \right) \partial_x g_2(0)
      \;\neq\; 0,
\\
b &=& \partial_{y_1}\partial_x f(0) \;\neq\; 0.
\end{array}
\end{equation}
In particular $h_0(0,0,z_2) = a z_2^3 + \Ord(|z_2|^4)$.
This is a cusp singularity. 
See \cite{GolubitskyGuillemin73-SingularityTheory, Gibson79-SingularityTheory,
Arnold94-CatastropheTheory, ArnoldGusejnZadeVarchenko85-Singularities, Murdock03-Unfoldings} 
for a background on singularity theory and its connection to dynamical systems.
In fact, non-degeneracies (\ref{Lie:eqDriftsingularity1dCoeffs})
allow to diffeomorphically transform (\ref{Lie:eqDriftsingularity1dExpansion}) into
the normal form
\begin{equation}\label{Lie:eqDriftsingularity1dNF}
 \dot{z_2} \;=\; \pm z_2^3 + z_1 z_2 + z_0 + \Ord(z_2^N),
\end{equation}
for arbitrary normal-form order $N$, see for example \cite{BruceGiblin92-Singularities}, proposition 6.10. 
This is a minimal versal unfolding of the cusp singularity.
See figure \ref{Lie:figCusp}.

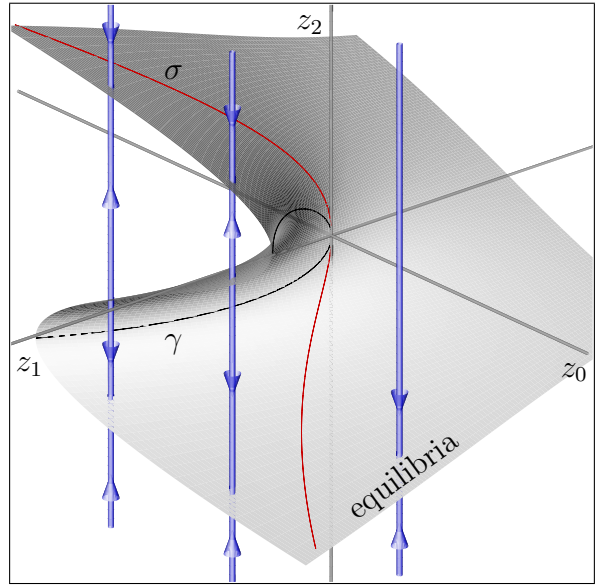
\begin{figure}
\centering
\setlength{\unitlength}{0.48\textwidth}
\begin{picture}(1.0,1.0)(0.0,0.0)
\put(0,0){\makebox(1,1){%
  \includegraphics[width=\unitlength]{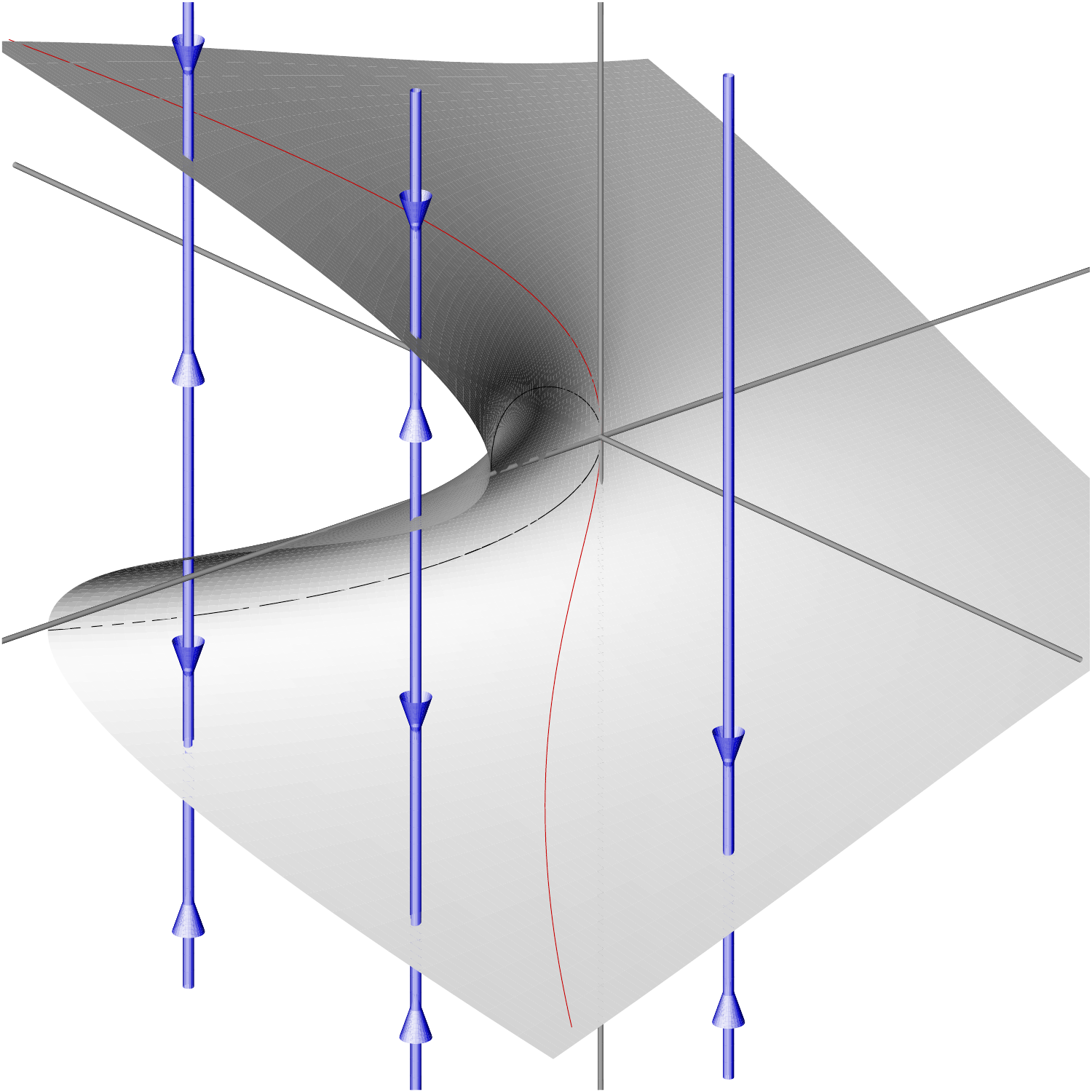}}}
\put(0.99,0.38){\makebox(0,0)[rt]{$z_0$}}
\put(0.01,0.39){\makebox(0,0)[lt]{$z_1$}}
\put(0.49,0.98){\makebox(0,0)[lt]{$z_2$}}
\put(0.28,0.43){\makebox(0,0)[t]{$\gamma$}}
\put(0.28,0.87){\makebox(0,0)[b]{$\sigma$}}
\put(0.60,0.10){\rotatebox{36}{\makebox(0,0)[bl]{equilibria}}}
\put(0,0){\framebox(1,1){}}
\end{picture}
\caption{\label{Lie:figCusp}
Cusp singularity $\dot{z_2} \;=\; a z_2^3 + z_1 z_2 + z_0$ 
with $a=-1$. Reverse direction of trajectories and signs of $z_0$, $z_1$ for $a=+1$.
The fold line $\gamma$ is connected by heteroclinic orbits to the curve $\sigma$, 
both curves have a common tangent at the origin.}
\end{figure}

Reverting the flow-box transformation, 
the cusp singularity yields a description of the local dynamics near 
a transcritical point with drift singularity on a two-dimensional 
manifold of equilibria. Note in particular the cusp-shaped fold line
\[
\gamma: \qquad z_1^3 \;=\; \mp \frac{27}{4} z_0^2 + \Ord(z_0^{N/3}), 
\qquad z_2^3 \;=\; \pm \frac{1}{2} z_0 + \Ord(z_0^{N/3})
\]
of the manifold of equilibria that is connected by heteroclinic orbits to 
the curve
\[
\sigma:\qquad z_1^3 \;=\; \mp \frac{27}{4} z_0^2 + \Ord(z_0^{N/3}), 
\qquad z_2^3 \;=\; \mp 4 z_0 + \Ord(z_0^{N/3}).
\]

\begin{prop}
Under condition (i-vi) the vector field (\ref{Lie:eqDriftsingularity}) 
in a local neighborhood $U$ of the origin 
is flow-equivalent to the cusp singularity (\ref{Lie:eqDriftsingularity1dNF}).
Depending on the sign of the cubic term, that is the sign of 
$a = \mathrm{sign\,} \left( \partial_{y_1} \partial_x f(0) \; \partial_{y_2} \partial_x g_1(0) 
 + \partial_{y_2}^2 \partial_x f(0) \; \partial_x g_2(0) \right)$,
all trajectories in $U$ converge to an equilibrium $(0,y)$ in forward time ($a=-1$) 
or backward time ($a=+1$).

In $U$, the transcritical points on the manifold of equilibria form a curve $\gamma$
through the origin. 
The unstable (for $a=-1$) and stable (for $a=+1$) sets, respectively,
of the two components $\gamma_1,\gamma_2$ of $\gamma\setminus\{0\}$ form manifolds of 
heteroclinic orbits on opposite sides of the manifold of equilibria. 
Their targets in forward time ($a=-1$) or backward time ($a=+1$) 
again form curves $\sigma_{1,2}$ on the manifold of equilibria with 
$\sigma_1\cup\{0\}\cup\sigma_2$ being a tangential curve to $\gamma$.
See figure \ref{Lie:figDriftSingularity}.
\end{prop}

\begin{remark}\label{Lie:remTranscritDriftSingularityNoEquilibria}
In contrast to the parameter-dependent drift singularity no equilibria bifurcate.
In fact, the drift non-degeneracy excludes any kind of recurrent or stationary 
orbits except the primary manifold of equilibria.
\end{remark}

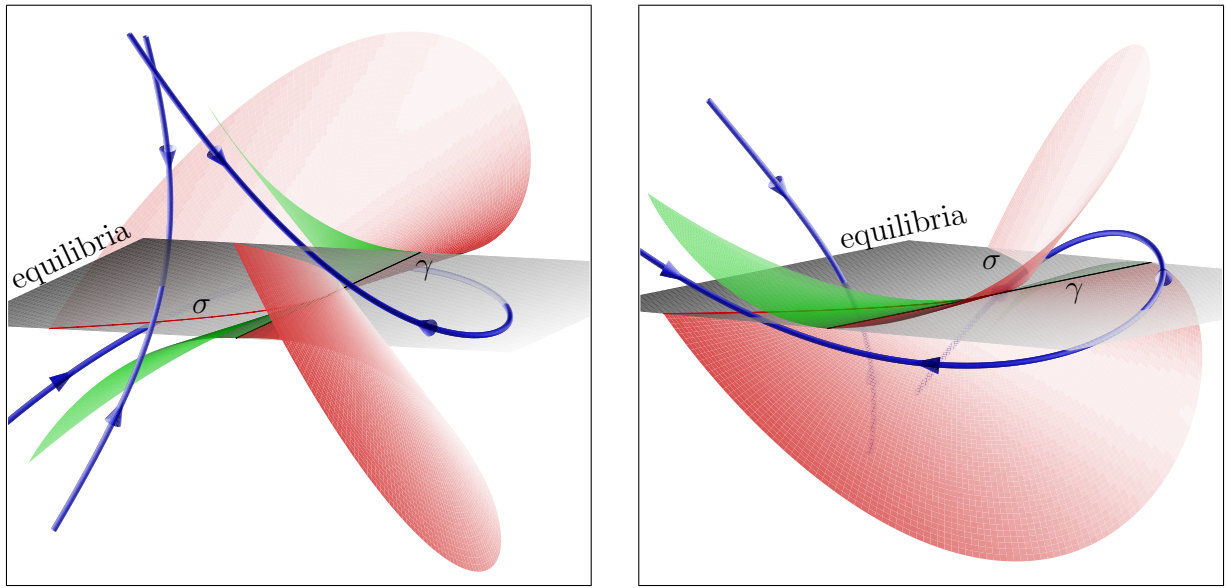
\begin{figure}
\setlength{\unitlength}{0.48\textwidth}
\begin{picture}(1.0,1.0)(0.0,0.0)
\put(0,0){\makebox(1,1){%
  \includegraphics[width=\unitlength]{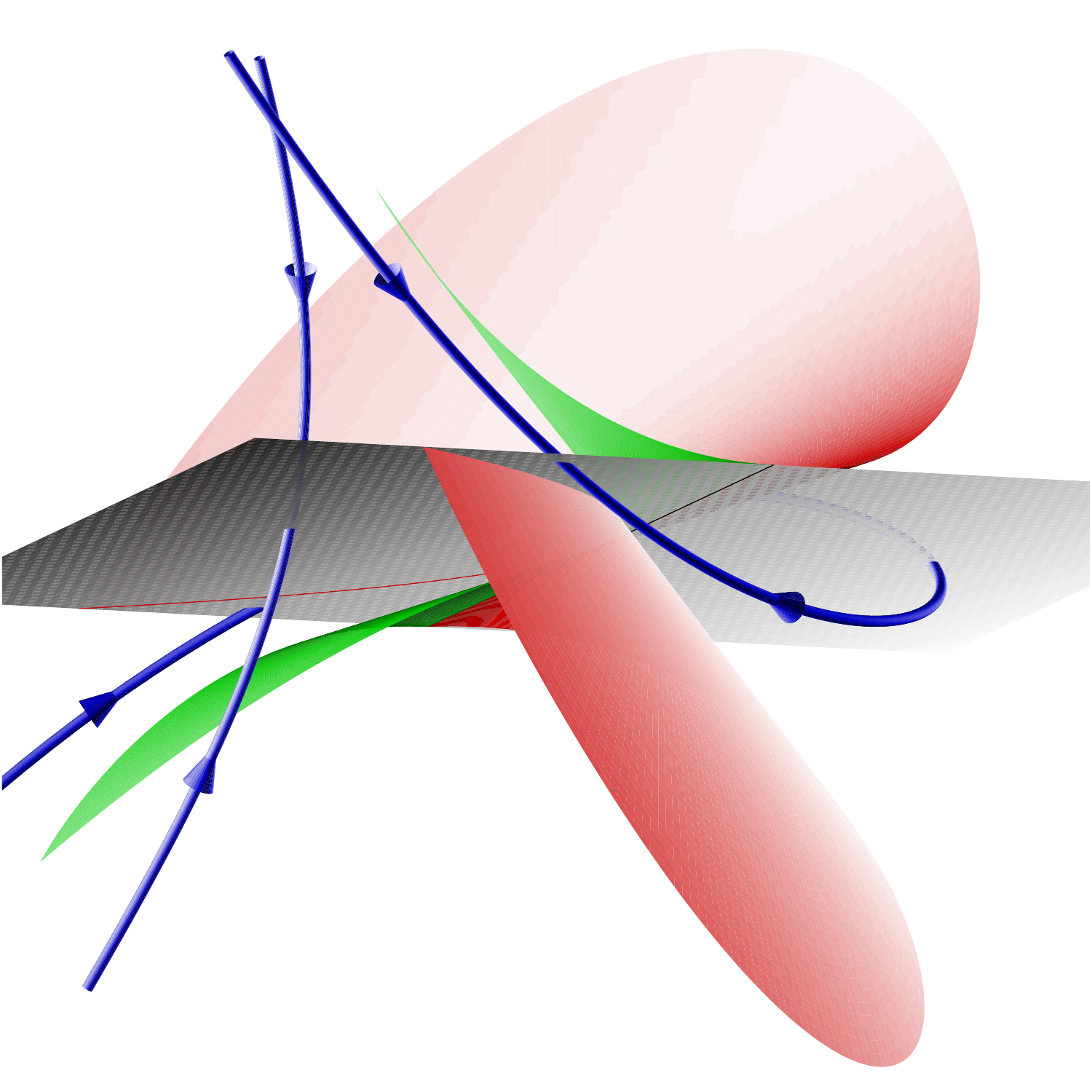}}}
\put(0,0){\framebox(1,1){}}
\put(0.02,0.50){\rotatebox{25}{\makebox(0,0)[bl]{equilibria}}}
\put(0.70,0.56){\makebox(0,0)[tl]{$\gamma$}}
\put(0.35,0.47){\makebox(0,0)[br]{$\sigma$}}
\end{picture}
\hfill
\begin{picture}(1.0,1.0)(0.0,0.0)
\put(0,0){\makebox(1,1){%
  \includegraphics[width=\unitlength]{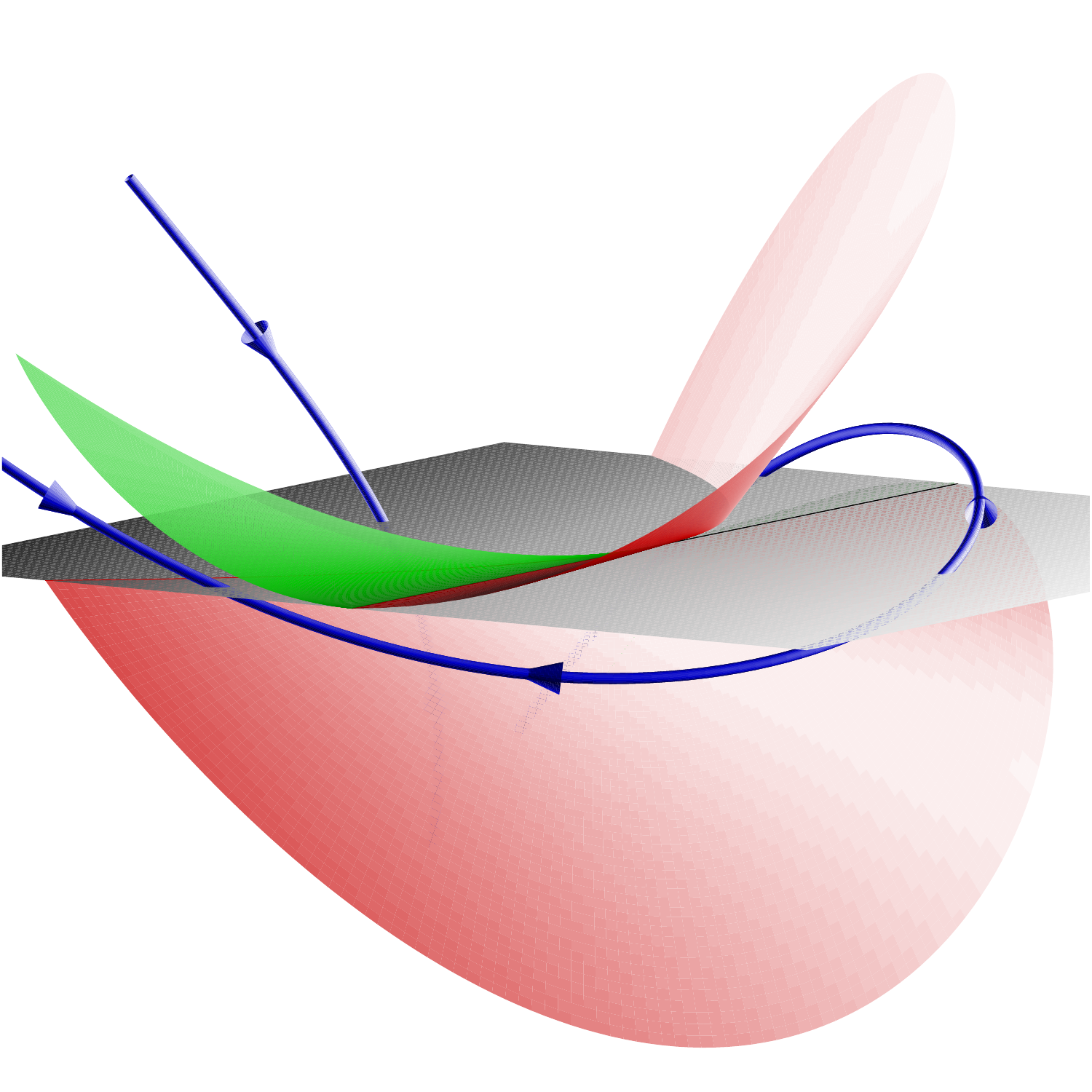}}}
\put(0,0){\framebox(1,1){}}
\put(0.35,0.57){\rotatebox{12}{\makebox(0,0)[bl]{equilibria}}}
\put(0.73,0.52){\makebox(0,0)[tl]{$\gamma$}}
\put(0.62,0.57){\makebox(0,0)[tr]{$\sigma$}}
\end{picture}
\caption{\label{Lie:figDriftSingularity}
  Transcritical point with drift singularity on a plane of equilibria. 
  Stable set of the line $\gamma$ of transcritical points in green, unstable set in red, 
  selected trajectories in blue. Two different views for $a=-1$. 
  Reverse direction of trajectories and switch colors of manifolds for $a=+1$.}
\end{figure}


\section{General Bifurcation at codimension-one manifolds}
\label{Lie:secGeneralTheorem}

We consider the general case of a manifold of equilibria of codimension one,
\begin{equation}\label{Lie:eqGeneralSingularity}
\left(\begin{array}{c} \dot{x} \\ \dot{y} \end{array}\right)
 \;=\;
 F(x,y)
 \;=\;
 \left(\begin{array}{c} f(x,y) \\ g(x,y) \end{array}\right),
\qquad
x \in \setR,\quad y \in \setR^m.
\end{equation}
Typically, such a system will arise as a reduced system on a center manifold of finite smoothness.
Following the discussion in the previous section we obtain the following theorem.

\begin{theorem}\label{Lie:thSingularity}
The exists a generic subset of the class of all smooth vector fields 
(\ref{Lie:eqGeneralSingularity}) with an equilibrium manifold $\{x=0\}$ of codimension one.
For every vector field in that class the following holds true:

At every point $(x=0,y)$ the vector field is locally flow equivalent to a $m$-parameter family 
\begin{equation}\label{Lie:eqodimEllSingularity}
\dot{z}_m \;=\; \pm z_m^{\ell+1} + \sum_{k=0}^{\ell-1} z_{k} z_m^k + \Ord(z_m^N),
\end{equation}
$0 \leq \ell \leq m$, of vector fields on the real line. 
Here $N$ is the arbitrary but finite normal-form order 
bounded by the smoothness of the initial vector field (\ref{Lie:eqGeneralSingularity}), 
$f,g\in\mathcal{C}^M$, $N \le M$, $N < \infty$.
This is a versal unfolding of the singularity $\dot{z}_m \;=\; \pm z_m^{\ell+1}$ at the origin.

In particular, near bifurcation points of codimension $m$, 
that appear robustly at isolated points on the equilibrium manifold, 
the vector field is locally flow equivalent to
\begin{equation}\label{Lie:eqodimMSingularity}
\dot{z}_m \;=\; \pm z_m^{m+1} + \sum_{k=0}^{m-1} z_k z_m^k + \Ord(z_m^N),
\end{equation}
i.e. a universal unfolding of the singularity $\dot{z}_m \;=\; \pm z_m^{m+1}$ at the origin.
\end{theorem}

\begin{proof}
The equilibrium condition $f(0,y)=g(0,y)=0$ for all $y\in\setR^m$ allows us to factor out $x$.
\begin{equation}\label{Lie:eqGeneralSingularityFactorX}
F(x,y) \;=\; x\tilde{F}(x,y)
       \;=\; x\left(\begin{array}{c} \tilde{f}(x,y) \\ \tilde{g}(x,y) \end{array}\right).
\end{equation}
The resulting vector field $\tilde{F}:\setR^{m+1}\to\setR^{m+1}$ does not vanish 
on the $m$-dimensional submanifold $\{x=0\}$, for generic $F$.
Without loss of generality, consider a neighborhood $U\subset\setR^{m+1}$ of the origin.

We can apply the flow-box theorem to $\tilde{F}$:
Take a local smooth section
\begin{equation}\label{Lie:eqGeneralTransverseSection}
\Sigma: \setR^m \supset V \longrightarrow U,
\end{equation}
through the origin, $\Sigma(0)=0$, transverse to the vectorfield $\tilde{F}$ in $U$. 
Let $\tilde{\Phi}_t$ be the flow generated by $\tilde{F}$.
Then the flow-box transformation 
\begin{equation}\label{Lie:eqGeneralFlowboxTransform}
h(z_0,...,z_m) \;=\; \tilde{\Phi}_{z_m}(\Sigma(z_0,...,z_{m-1}))
\end{equation}
transforms $\tilde{F}$ into the constant vector field 
$[\Diff h]^{-1}(\tilde{F} \circ h) = (0,\ldots,0,1)$.
Again, $\tilde{\Phi}_t$ denotes the flow to the vector field $\tilde{F}$.
Applying the transformation $h$ to the vector field $F|_U$, 
we obtain a $m$-parameter family 
$[\Diff h]^{-1}(F \circ h) = (0,\ldots,0,\pi_x h)$ 
of vector fields on the real line in a neighborhood $V$ of the origin.

Classification of germs of vector fields and their versal unfoldings is the 
topic of singularity or catastrophe theory.

Singularities on the real line have the form $\dot{z}_m \;=\; \pm z_m^{\ell+1}$. 
In generic $m$-parameter families at most $m+1$ 
leading coefficients of the Tailor expansion vanish, i.e. $\ell \leq m$ and
\[
\dot{z}_m \;=\; \pm z_m^{\ell+1} + \sum_{k=0}^{\ell-1} \zeta_k(z_0,...,z_{m-1}) z_m^k + \Ord(z_m^{\ell+2}).
\]
The coefficient $\zeta_\ell$ vanishes by linear transformation of $z_m$ and the map
$(z_0,...,z_{m-1}) \mapsto (\zeta_0,...,\zeta_{\ell-1})$ has full rank, generically.
Remainder terms, $\Ord(z_m^{\ell+2})$, can be pushed to any finite normal-form order, 
by a suitable coordinate change.
This yields system (\ref{Lie:eqodimEllSingularity}). 
See also \cite{BruceGiblin92-Singularities}, chapter 6.

Genericity conditions amount to algebraic conditions of the coefficients of the 
Taylor expansion at the origin. 
These conditions translate via (\ref{Lie:eqGeneralFlowboxTransform}) 
to generic conditions on $F$.

The versal unfolding (\ref{Lie:eqodimEllSingularity}), one the other hand, 
is a system of the form (\ref{Lie:eqGeneralSingularity}). 
Therefore, it represents the versal unfolding of a generic singularity along 
$m$-dimensional manifolds of equilibria in $(m+1)$-dimensional phase space.
\end{proof}


\section{Discussion}
\label{Lie:secDiscussion}

The present result is a first step towards a more systematic treatment 
of bifurcations without parameters
than done by case studies in \cite{FiedlerLiebscherAlexander98-HopfTheory, 
FiedlerLiebscher01-TakensBogdanov, AfendikovFiedlerLiebscher07-PlaneKolmogorovFlows}.

The removal of the manifold of equilibria by a scalar, albeit singular, multiplier greatly facilitates 
the analysis but restricts it to the case of manifolds of codimension one in the phase space, 
see (\ref{Lie:eqDriftsingularityFactorX}) and (\ref{Lie:eqGeneralSingularityFactorX}).
Hopf points and Bogdanov-Takens points require manifolds of equilibria of at least codimension two. 
Their analysis in \cite{FiedlerLiebscherAlexander98-HopfTheory, FiedlerLiebscher01-TakensBogdanov}
uses a blow-up or rescaling procedure reminiscent of the scalar multiplier used here.
It seems worthwile to closer connect these bifurcations without parameters to singularity theory.
This might provide a suitable setting to also include singularities of the set of equilibria 
and generalize the manifold to varieties.

Contrary to classical bifurcation theory, no recurrent dynamics has been found so far
near bifurcation points without parameters.
For the codimension-one manifolds of equilibria discussed here, 
the drift nondegeneracy yielding the flow-box transformation prevents any recurrent dynamics. 
Similar drift conditions hold true at generic Hopf and Bogdanov-Takens points. 
In fact this is the drift which distinguishes bifurcations without parameters 
from classical bifurcations 
by preventing any flow-invariant transverse foliation. 
Recurrent dynamics should be possible at bifurcations points of higher codimension 
as the drift condition gets less restrictive.


\clearpage
\bibliographystyle{alpha_abbrv.bst}
\bibliography{Lie10-Codim1EqManifolds}


\end{document}